\documentclass{amsart}
\usepackage{graphicx}
\usepackage[usenames]{color}
\usepackage[normalem]{ulem}
\usepackage{epigraph}
\usepackage{cancel}
\newtheorem{thm}{Theorem}[section]

\theoremstyle{definition}
\newtheorem{defn}[thm]{Definition}
\theoremstyle{remark}

\numberwithin{equation}{section}

\newcommand{\abs}[1]{\left\vert#1\right\vert}

\begin{document}

\title[Applications of Zalcman's lemma ]{Applications of Zalcman's lemma in $C^N$}%
\author{P.V. Dovbush}%
\address{Institute of Mathematics and Computer Science of Moldova, 5 Academy  Street,  Kishinev, MD-2028,
{Republic of Moldova}}%
\email{peter.dovbush@gmail.com}%



\begin{abstract}
{The aim of this paper is to give  some applications of Zalcman's Rescalling Lemma.}
\end{abstract}
\maketitle
\section{introduction}
There is a general principle that has been used as a guide to the study
of normality-often called Bloch's principle. The idea is that if a property
forced on entire functions causes the function to be constant, then that
property on a family of functions on a region forces the family to be normal.
 The name ''Bloch's principle'' comes from
the statement in his paper  (the original quote is in Latin): "There is
nothing in the infinite which did not exist before in the finite."

Of course, the principle is not a rigorous statement, since the term ''property'' that used in the
formulation of Bloch's principle is not precisely defined. In fact, the are various counterexample
to this principle.

In his article \cite{[ZL1]} from 1998, Lawrence Zalcman modestly says that he proved his
''little lemma'' to give Bloch's principle precise form.  Earlier,
Lohwater-Pomme-\-renke \cite{[LP]} had used a similar construction but for single functions on a unit disc
obeying $\sup_{\{|z|<1\}}(1 - |z|^2)f^\sharp(z) < \infty$ (such functions are
called normal functions) rather than families. Pommerenke  got the idea from reading Landau \cite{[ELA]}. Landau in turn credits Valiron \cite{[VAG]} with this idea. Of course, what can be found in Valiron's paper is still far away from the statement of the lemma in Zalcman's paper.
The general idea of multiplying by $(1 - |z|^2)$ used in the proof of
Zalcman's lemma is due to Landau \cite{[ELA1]}.

Not surprisingly, Zalcman's lemma
generated considerable followup, summarized in Zalcman \cite{[ZL1]}.
For a detailed history of the Zalcman lemma, it extension, and it applications up to 2017, see Steinmetz \cite{[SN]}.

Let us illustrate the use of Zalcman's Rescalling Lemma by showing  how it can be used to derive Montel's
Theorem in several complex variables see also \cite{[ZL2]}.  First of all we  shall need  one auxiliary proposition for the proof of which see, for example, \cite[Corollary,  p.80]{[NRS]}.
\begin{thm} \label{[HL]} (Hurwitz's theorem) Let $\Omega$ be an open connected set
in $C^n$ and $\{g_j(z)\}$ a sequence of holomorphic functions on $\Omega,$ converging
uniformly on compact sets to a holomorphic function $g.$ Then if
$g_j (z) \neq 0$ for all $j ,$ and all $z,$ and $g$ is nonconstant, we have $g(z)\neq 0$
for all $z \in \Omega.$
\end{thm}
The theorem  of Montel is easy
consequence of  Zalcman's  lemma \cite{[DP]}.

\begin{thm} \label{{MT}} (Montel's theorem) Let $\mathcal F$ be a family of holomorphic functions on an
open set $\Omega \subseteq C^n$  that  omit  two fixed  complex  values. Then, each sequence of functions in $\mathcal F$ has a subsequence which
converges uniformly on compact subsets.
\end{thm}
\begin{proof} Composing the functions of $\mathcal F$ with a linear fractional transformation, we may
also assume that the omitted values are $0$ and $1.$  Suppose $\mathcal F$ is not normal on $\Omega.$ Then by Zalcman's Rescalling Lemma \cite{[DP]}, there exist $f_j \in \mathcal F,$ $z_j\in \Omega$
and $\rho_j \to 0+$ such that $f_j(z_j + \rho_j\zeta)=g_j(\zeta) \to g(\zeta)$ uniformly on compact subsets of $C^n,$
where $g$ is a nonconstant entire function. By  Theorem \ref{[HL]}, $g$ does not take
on the values $0$ and $1,$ since no $f_j$ does. Let $b\in C^n.$ The function
$$g_b(\lambda):=g(\lambda\cdot b)$$
is entire function on $C,$ satisfies
$$ g_b(0)=g(0), \ \ g_b(1)=g(b)$$
and
$$ g_b(C)\subset g(C^N)\subseteq C\setminus \{0,1\}.$$
But then, by one-dimensional version of Picard's Little Theorem, $g_b$ is
constant, hence $g(0)=g(b)$ for all $b\in C^n,$  a contradiction.
\end{proof}

\emph{We sketch a simple proof of the theorem of Montel, due to A.
Ros \cite[p. 218]{[ZL1]}.} Since normality is a local notion, we may suppose that $\Omega = B,$ the unit
ball. So let $\mathcal F$ be as in the statement of Montel's Theorem and
suppose that $\mathcal F$ is not normal. Composing with a linear fractional transformation, we may also assume that
the omitted values are
$ 0$ and $ 1.$  This implies that if $f \in \mathcal F$ and $m\in N$ there exists a
function $g$ holomorphic in $\Omega$ such that $g^{2^m} = f.$ Let $\mathcal F_m$ be the family of all such
functions $g.$ Note that
$$ \Big| \frac{1}{2^m}\Big|^2\Big(\frac{|f|^{-1}+|f|}{|f|^{-1/{2^m}}+|f|^{1/{2^m}}}\Big)^2L_z(\log(1+|f|^2,v)=L_z(\log(1+|g|^2,v) $$ for all
$(z,v)\in \Omega\times C^n.$
This implies that
$$  g^\sharp(z)=\frac{1}{2^m}\frac{|f|^{-1}+|f|}{|f|^{-1/{2^m}}+|f|^{1/{2^m}}}f^\sharp(z)\geq\frac{1}{2^m}f^\sharp(z) \ \ (z\in \Omega),$$
where we have used the inequality $a^{-1}+a \geq a^{-t}+a^t $ valid for $a > 0$ and $0 < t < 1.$
By Marty's Criterion, the family $ \{f^\sharp : f \in \mathcal F\}$ is not locally bounded. We deduce
that, for fixed $m \in N,$ the family  $ \{g^\sharp : g \in \mathcal F_m\}$  is not locally bounded. Using
Marty's Criterion again we find that $\mathcal F_m$ is not normal, for all $m \in N.$
Note that if $g \in \mathcal F_m,$ then $g$ omits the values $e^{2\pi i k/2m}$ for $k \in Z.$ From the Zalcman
Principle we thus deduce that there exists an entire
function $g_m$ omitting the values $e^{2\pi i k/2m}$ and satisfying $g_m(z)\leq g_m^\sharp(0) = 1.$ The $g_m$
thus form a normal family and we have $g_{m_j}\in G$ for some subsequence $\{g_{m_j}\}$
the values $e^{2\pi i k/2m}$ for all $k, m\in N.$ Since $G(C^n)$ is open this implies that $|G(z)| \neq 1$
for all $z \in C^n.$ Thus either $|G(z)| < 1$ for all $z \in C^n$ or $|G(z)| > 1$ for all $z \in C^n.$
In the first case $G$ is bounded and thus constant by Liouville's Theorem. In the
second case $1/G$ is bounded. Again $1/G$ and thus $G$ is constant. Thus we get a
contradiction in both cases.

The theorem of Schottky is an easy consequence of  the preceding work.

\begin{thm}
    (Schotthy's theorem).
    Let $M > 0$ and $r\in [0,1)$ be given.
    Then there exists a constant $C > 0$ depending only on $r$ and $f(0)$ such that the following implication holds:

    If $f$ is holomorphic in the unit ball $B$, omits $0$ and $1$ from its range, and if $|f(0)|<M$, then $|f(z)|<C$ for all $z\in B(0,r)= {z\in  C^n :|z|<r}$.
\end{thm}
\begin{proof}
    Let denote $\mathbf{G}$ denote  the  family  of  functions  holomorphic  on
$B$ that omit $0$ and $1$.
    By Montel's theorem the family $\mathbf{G}\circ Aut(B)$ is normal.
    Marty's characterization of normal families yields a non-negative constant $L > 0$ such that for any $f\circ \varphi\in \mathbf{G}\circ Aut(B)$
    \[
        \frac{(f\circ \varphi)'(0)}{1+\abs{f\circ \varphi(0)}^2}<L.
    \]
    The constant $L$ does not depend on  $f \in \mathbf{G}$.
    Since $f$ is a normal function in $B$  we have
    \[
        L_z(\log (1+|f|^2),v)\leq L^2(F_{B}(z,v))^2\ \ \textrm{ for all }z\in B, v\in  C^n.
    \]

    Applying Schwarz's lemma to the restriction of the map $g: B\to \Delta$ to the complex line passing through the point $z\in B$ in the direction of the vector $v\in C^n$, it is not difficult to get the estimate
    \[
        F^C_{B}(z,v)\leq \frac{\abs{v}}{1-\abs{z}}.
    \]
    It follows
    \[
        L_z(\log (1+|f|^2),v)\leq L^2 \frac{\abs{v}^2}{(1-\abs{z})^2}\ \ \textrm{ for all }z\in B, v \in C^n.
    \]
    Taking the supremum over all $v\in C^n$ such that $\abs{v}=1$, gives
    \[
        \frac{|Df(z)|}{1+|f(z)|^2}<\frac{L}{1-r}\textrm{ for all } z\in \overline{B(0,r)}, \ \ r\in [0,1).
    \]
    Since $f$ never takes the value $0$, the function $t \to |f(tz)|$ is continuously differentiable on $(0,1)$ for any fixed $z \in B(0, r)$ and
    \[
        \frac{d}{dt}\Big(\arctan(|f(tz)|)\Big)\leq\frac{|(Df(tz),\overline{z})|}{1+|f(tz)|^2}< \frac{|Df(tz)||\overline{z}|}{1+|f(tz)|^2}<\frac{L}{1-r}.
    \]
    Then we deduce for $t\in [0,1]$
    \[
        \Big|\arctan(|f(z)|)-\arctan(|f(0)|)\Big|<\int_0^{1}\Big|\frac{d}{dt}(\arctan(|f(tz)|))\Big|dt<\frac{L}{1-r}
    \]
    and so
    \[
        \arctan(|f(z)|)<\arctan(|f(0)|)+\frac{L}{1-r}<\arctan M+\frac{L}{1-r}
    \]
    which is the theorem with $C=\tan(\arctan M+L/(1-r))$.
\end{proof}

We shall now show that the following result, in which the values
omitted are allowed to vary with the function, as long as they do not approach one
another too closely, is also true.

\begin{thm} (Carath\'{e}odory's theorem)
Let $\mathcal F$ be a family of functions holomorphic on $\Omega\subset C^n.$ Suppose that
for some $\varepsilon>0,$ there exist for each $f \in \mathcal F$ distinct points $a_f, b_f\in C $ such that
for all $z \in \Omega,$ $f(z) \neq a_f,b_f$ and
$$ s(a_f,\infty)s(a_f,b_f)s(\infty,b_f)>\varepsilon.$$
Then $\mathcal F$ is normal on $\Omega.$
\end{thm}
 \begin{proof} Otherwise, there exists some ball $B$ in $\Omega,$ which we may assume to be unit ball, on
which $\mathcal F$ fails to be normal. Then by Zalcman's Rescalling Lemma \cite{[DP]}, there exist $f_j \in \mathcal F,$ $z_j\in B$
and $\rho_j \to 0+$ such that $f_j(z_j + \rho_j\zeta)=g_j(\zeta) \to g(\zeta)$ uniformly on compact subsets of $C^n,$
where $g$ is a nonconstant entire function.  Taking successive subsequences and
renumbering, we can assume that $ a_{f_j} \to a$ and $ b_{f_j} \to b,$  where $a$ and
$b$ are distinct points in $C.$ Since $g_j(\zeta) - a_{f_j}\neq 0$ and $g$ is nonconstant, it follows
from  Theorem \ref{[HL]} that $g(\zeta) \neq a.$ Similarly, $g(\zeta)\neq b.$ Again, we may assume that the omitted values are $0$ and $1.$
The reasoning used at the end of Theorem \ref{{MT}} shows that $g$ is a constant, a contradiction.
\end{proof}

Montel's theorem remains valid if the omitted values are replaced by omitted
functions, so long as the omitted functions never take on the same value at points of $\Omega.$

\begin{thm}  (Fatou's theorem)
Let $a(z)$ and $b(z)$ be functions holomorphic on $\Omega \subset C^n$ such that
$a(z) \neq b(z)$
for each $z \in \Omega.$ Let $\mathcal F$ be a family of functions holomorphic on $\Omega$ such that for each
$z \in \Omega$
$$f(z) \neq a(z) \qquad f(z) \neq b(z) $$
for all $f\in \mathcal F.$ Then $\mathcal F$ is normal on $\Omega.$
\end{thm}
\begin{proof} Consider the family of functions
$$\mathcal G= \Big\{\frac{f(z) - a(z)}{f(z) - b(z)} \textrm{ for all } f\in \mathcal F \Big\}.$$
Then each $g \in \mathcal G$ is holomorphic on $\Omega;$ and if $g \in  \mathcal G,$ then $g(z) \neq 0,1$ for $z \in \Omega.$
Thus $\mathcal G$ is normal on $\Omega$  by  Theorem \ref{{MT}}. But then, as is easily seen, $\mathcal F$ is normal
on $\Omega$ as well.
\end{proof}

\begin{defn} Let $g(\lambda)$ be an entire (holomorphic in $C$) function, if  the  equation  $g(\lambda) =a,$ $a\in C,$ has  no  simple roots then $a$ called a totally
ramified values.
\end{defn}

 Note that an omitted value trivially satisfies this definition,
but that it will be useful to distinguish between omitted values and non-omitted
totally ramified values.

\begin{thm} (R. Nevanlinna) \cite[Theorem 17.3.10., p. 274]{[BAS]}
Let $g$ be an entire  function. Then $g$ has at
most two totally ramified (finite) values.
\end{thm}

Let $f$ be a holomorphic function on an open connected set $\Omega$ in  $C^n.$  Define the value set by
$$A_f(a)=\{z\in \Omega : f(z)=a\}= f^{-1}[\{a\}].$$

 \begin{thm}  (Hurwitz's theorem). \cite[Corollary p.80]{[NRS]} Let  $\Omega$  be an open connected set
in  $C^n$  and let $\{f_j\}$ be a  sequence of holomorphic functions  on  $\Omega,$  converging
uniformly on compact sets to a nonconstant holomorphic function  $f.$  If
$A_{f_j}(a)=\emptyset$  for all  $j  $    then  $A_f(a)=\emptyset.$
\end{thm}
Let $\zeta,v\in C^n,$ $v\neq 0.$ The set
$$\{\xi\in C^n : \xi=\zeta+\lambda v, \lambda\in C\}$$
is called a complex line in $C^n.$

The restriction of an entire (holomorphic in $C^n$) function $g$ to a complex line $\{\xi=\zeta+\lambda v, \lambda\in C\}$ clearly is an entire
$g(\zeta+\lambda\cdot v)$ of complex variable $\lambda$ in $C.$

Mimic the proof given in the one-dimensional case in \cite[Theorem p. 219]{[ZL1]} we can prove  the following normality criterion. The original result
due to Lappan \cite{[PLA]}  corresponds to normal functions.

\begin{thm} \label{LT} family $\mathcal F$  of holomorphic functions on a domain $\Omega\subset C^n$  is normal
on $\Omega$  if and only if for each compact set $K \subset \Omega,$ there exists a set
$E = E(K) \subset C$ containing at least three  distinct values and a finite constant $M=M(K)>0$ for which
\begin{equation}\label{e1}
f^\sharp(z) \leq M \ \ z\in K, f(z)\in E
\end{equation}
for all  $f\in \mathcal F.$
 \end{thm}
\begin{proof} Marty's Theorem shows that (\ref{e1}) is necessary with $E = C.$

To prove sufficiency, suppose that (\ref{e1}) holds but $\mathcal F$ is not normal at point $z_0\in \Omega.$ Since all features of the theorem are local, and translation and scale-change invariant, it suffices to consider the case when $z_0=0$ and $\Omega={B(0,1)}.$  If $\mathcal F$ is not normal at $0,$ it follows from Zalcman's lemma \cite[Theorem 3.1]{[DP]}  that there exist $f_j\in \mathcal F,$ $z_j\to 0,$ $\rho_j\to 0,$
 such that the  sequence $$g_j(z):=f_j(z_j+\rho_j z) $$
 converges uniformly  on compact subsets of  $C^n$ to a non-constant entire function $g$ satisfying $g^\sharp(z)\leq g^\sharp(0)=1.$

 Let $K \subset B(0,1)$ be a closed ball in $C^n$ about $0.$ Let $E(K)\supset\{a_1,a_2,a_3\},$ where $a_1,a_2,a_3$ are distinct values in $C.$

  Suppose that $\zeta^l \in A_g(a_l).$ Choose $R>0$ such that  $\zeta^l\in S_R=\{\xi\in C^n : |\xi|<R\}.$ Since normality is a local property, the restriction of a family $\{g_j-a_l\}$  to any open ball $s_n(\zeta^l):=\{\xi\in C^n : |\xi-\zeta^l|<1/n\}\subset S_R$  is a normal family.
 By Hurwitz' theorem to $A_{g_j}(a_l)\cap s_n(\zeta^l)\neq\emptyset$   for $j$ sufficiently large since $g$ is not a constant function.
It is routine to show that there exist a sequence $\{p_{j}^l\}\subset S_K,$   such that  $g_j(p_{j}^l)\to a_l.$

Since $\rho_j\to 0$ we see that $z_j+\rho_j p_{j}^l\in K$ for $k$ sufficiently large.
  Now by (\ref{e1}),
$f^\sharp_{j}(z_{j} + \rho_{j}p_{j}^l) \leq M $ for $k$ sufficiently large, so that
$$ g^\sharp(\zeta^l) = \lim_{k\to \infty}g^\sharp_{j}(p_{j}^l)
= \lim_{k\to \infty}\rho_{j}f^\sharp_{j}(z_{j} + \rho_{j}p_{j}^l)\leq \lim_{k\to \infty}\rho_{j}M=0.$$
Thus $g^\sharp(\zeta^l) = 0.$

Since
$$g^\sharp(\zeta^l) =\max_{\{v\in C^n: |v|=1\}} \frac{|\frac{d}{d\lambda}g(\zeta^l+\lambda\cdot v)|_{\lambda=0}|^2}{1+|g(\zeta^l)|^2}=0$$
it follows $a_l$ is a finite totally ramifies value for $g(\zeta^l+\lambda\cdot v)$ ( $v\in C^n$ be  arbitrary (but fixed) and $|v|=1$).

If $\{\xi\in C^n : \xi=\zeta^l+\lambda\cdot v\}\cap A_g(a^k)\ni \zeta^k $ then $\zeta^k=\zeta^l+\lambda_k\cdot v$ for some $\lambda_k\in C.$
Arguing as above we have
$$ g^\sharp(\zeta^k)= g^\sharp(\zeta^l+\lambda_k\cdot v)=0.$$
Hence $a_k$ is a totally ramified value for $g(\zeta^l+\lambda\cdot v).$

If $\{\xi\in C^n : \xi=\zeta^l+\lambda\cdot v\}\cap A_g(a^k)=\emptyset$ then $a^k$ is omitted value for $g(\zeta^l+\lambda\cdot v)$ and hence a totally ramified (finite) value of the function $g(\zeta^l+\lambda\cdot v).$

Thus  $a_1, a_2, a_3$ are  three totally ramified (finite) values for the entire function
$g\Big(\zeta^l+\lambda\cdot v\Big).$
By Nevanlinna's theorem $g\Big(\zeta^l+\lambda\cdot v\Big)$  is constant.
 Since $v$ was arbitrary the function $g$ is constant, a contradiction.
\end{proof}

\end{document}